\documentclass[10pt]{amsart}

\usepackage{amssymb,latexsym,amsmath}
\usepackage{graphics}                 
\usepackage{color}                    
\usepackage{hyperref}                 
\usepackage[all]{xy}

\begin{document}

\newtheorem{theorem}{Theorem}[section]
\newtheorem{proposition}[theorem]{Proposition}
\newtheorem{lemma}[theorem]{Lemma}
\newtheorem{corollary}[theorem]{Corollary}
\newtheorem{question}[theorem]{Question}
\newtheorem{remark}[theorem]{Remark}
\newtheorem{example}[theorem]{Example}
\newtheorem{conjecture}[theorem]{Conjecture}
\newtheorem{definition}[theorem]{Definition}
\newtheorem{correction}{Correction}

\def\A{{\mathbb{A}}}
\def\C{{\mathbb{C}}}
\def\H{{\mathbb{H}}}
\def\L{{\mathbb{L}}}
\def\P{{\mathbb{P}}}
\def\Q{{\mathbb{Q}}}
\def\Z{{\mathbb{Z}}}

\def\Ch{{\rm Ch}}
\def\id{{\rm id}}
\def\sp{{\rm SP}}

\def\mC{{\mathcal{C}}}
\def\mP{{\mathcal{P}}}
\def\mZ{{\mathcal{Z}}}

\title{Lawson Homology for Abelian Varieties}
\author{Wenchuan Hu}
\date{September 27, 2011}
\date{October 16, 2011}

\address{
School of Mathematics\\
Sichuan University\\
Chengdu 610064 \\
P. R. China
}

\address{\emph{Current Address}:
Vivatsgasse 7\\
53111 Bonn\\
Germany
}

\email{huwenchuan@gmail.com}

\begin{abstract}
In this paper we introduce the Fourier-Mukai transform for Lawson homology of abelian varieties and prove an inversion theorem for
the Lawson homology as well as the  morphic cohomology of abelian varieties. As applications, we obtain the direct sum decomposition
of the Lawson homology and the morphic cohomology groups with rational coefficients,
inspired by Beauville's works on the Chow theory. An analogue  of
the Beauville conjecture for Chow groups is proposed and is shown to be equivalent to the (weak) Suslin conjecture for Lawson homology.
A  filtration on Lawson homology is proposed  and conjecturally it coincides to  the filtration given by the direct sum decomposition
of Lawson homology for  abelian varieties. Moreover, a refined Friedlander-Lawson duality theorem  is obtained for abelian varieties. 
We summarize several related conjectures in Lawson homology theory in the appendix for convenience.
\end{abstract}

\maketitle
\tableofcontents

\section{Notations}
In this paper, all varieties are defined over the complex number field $\C$. Let $X$ be a projective variety of dimension $n$.
Denoted by $\mZ_p(X)$ the space of algebraic $p$-cycles on $X$. Let $\Ch_p(X)$ be the Chow group of $p$-cycles on $X$, i.e.
$\Ch_p(X)=\mZ_p(X)/{\rm \{rational ~equivalence\}}$. Set $\Ch_p(X)_{\Q}:=\Ch_p(X)\otimes \Q$, $\Ch_p(X)=\bigoplus_{p\geq0} \Ch_p(X)$
and $\Ch_*(X)_{\Q}=\bigoplus_{p\geq0} \Ch_p(X)_{\Q}$. Let $A_p(X)$ be the space of $p$-cycles on $X$ modulo
the algebraic equivalence, i.e. $A_p(X)=\mZ_p(X)/{\sim_{alg}}$, where $\sim_{alg}$ denotes the algebraic equivalence.
Set $A_p(X)_{\Q}:=A_p(X)\otimes \Q$, $A(X)=\bigoplus_{p\geq0} A_p(X)$ and $A_*(X)_{\Q}=\bigoplus_{p\geq0} A_p(X)_{\Q}$.


\section{Lawson homology}

The \emph{Lawson homology}
$L_pH_k(X)$ of $p$-cycles for a projective variety is defined by
$$L_pH_k(X) := \pi_{k-2p}({\mathcal Z}_p(X)) \quad for\quad k\geq 2p\geq 0,$$
where ${\mathcal Z}_p(X)$ is provided with a natural topology (cf.
\cite{Friedlander1}, \cite{Lawson1}). It has been extended to define for a quasi-projective  variety by Lima-Filho (cf. \cite{Lima-Filho})
and Chow motives (cf. \cite{Hu-Li}). For general background, the
reader is referred to Lawson' survey paper \cite{Lawson2}.
The definition of Lawson homology has been extended to negative integer $p$. Formally for $p<0$, we have
$L_pH_k(X)=\pi_{k-2p}({\mathcal Z}_{0}(X\times \mathbb{C}^{-p}))=H^{BM}_{k-2p}(X\times\mathbb{C}^{-p})=H^{BM}_k(X)=L_0H_k(X)$
(cf. \cite{Friedlander-Haesemesyer-Walker}), where $H^{BM}_{*}(-)$ denotes the Borel-Moore homology.

In \cite{Friedlander-Mazur}, Friedlander and Mazur showed that there
are  natural transformations, called \emph{Friedlander-Mazur cycle class maps}
\begin{equation}\label{eq01}
\Phi_{p,k}:L_pH_{k}(X)\rightarrow H_{k}(X)
\end{equation}
for all $k\geq 2p\geq 0$.

Recall that Friedlander and Mazur constructed a map called the $s$-map $s:L_pH_k(X)\to L_{p-1}H_k(X)$ such that the cycle class map
$\Phi_{p,k}=s^p$ (\cite{Friedlander-Mazur}).
Explicitly, if $\alpha\in L_pH_k(X)$ is represented by the homotopy class of a contiunous map  $f:S^{k-2p}\to \mZ_p(X)$, then
$\Phi_{p,k}(\alpha)=[f\wedge S^{2p}]$, where $S^{2p}=S^2\wedge\cdots\wedge S^2$ denotes the $2p$-dimensional topological sphere.

Set
{$$
\begin{array}{llcl}
&L_pH_{k}(X)_{hom}&:=&{\rm ker}\{\Phi_{p,k}:L_pH_{k}(X)\rightarrow
H_{k}(X)\};\\
&L_pH_{k}(X)_{\Q}&:=&L_pH_{k}(X)\otimes\Q;\\
&T_pH_{k}(X)&:=&{\rm Image}\{\Phi_{p,k}:L_pH_{k}(X)\rightarrow
H_{k}(X)\};\\
&T_pH_{k}(X,{\mathbb{Q}})&:=&T_pH_{k}(X)\otimes {\mathbb{Q}}.
\end{array}
 $$}

Denoted by $ \Phi_{p,k,\Q}$ the map $ \Phi_{p,k}\otimes{\Q}:L_pH_{k}(X)_{\Q}\rightarrow H_{k}(X,\Q) $.
The \emph{Griffiths group} of dimension $p$-cycles is defined to be
$$
{\rm Griff}_p(X):={\mathcal Z}_p(X)_{hom}/{\mathcal Z}_p(X)_{alg}.$$

Set
$$
\begin{array}{lcl}
{\rm Griff}_p(X)_{\Q}&:=&{\rm Griff}_p(X)\otimes\Q;\\
{\rm Griff}^q(X)&:=&{\rm Griff}_{n-q}(X);\\
{\rm Griff}^q(X)_{\Q}&:=&{\rm Griff}_{n-q}(X)_{\Q}.
\end{array}
$$

It was proved by Friedlander \cite{Friedlander1} that, for any
smooth projective variety $X$, $$L_pH_{2p}(X)\cong {\mathcal
Z}_p(X)/{\mathcal Z}_p(X)_{alg}=A_p(X).$$

Therefore
\begin{eqnarray*}
L_pH_{2p}(X)_{hom}\cong {\rm Griff}_p(X).
\end{eqnarray*}

For any smooth quasi-projective variety
$X$, there is an intersection pairing (cf. \cite{Friedlander-Gabber})  $$L_pH_k(X)\otimes L_qH_l(X)\to L_{p+q-n}H_{k+l-2n}(X),$$
induced by the diagonal map $\Delta: X\to X\times X$. More precisely, the composition $\mZ_p(X)\times \mZ_q(X)\stackrel{\times}{\to}\mZ_{p+q}(X\times X)\stackrel{\Delta^{!}}{\to} Z_{p+q-n}(X)$, where $\times$ is the Cartesian
product of cycles and $\Delta^{!}$ is the Gysin map, factors through $\mZ_p(X)\wedge \mZ_q(X)$. On the level of homotopy groups we have intersection pairing
$$
\pi_{k-2p}(\mZ_p(X))\otimes\pi_{l-2q}(\mZ_q(X))\stackrel{\bullet}{\to}\pi_{k+l-2(p+q)}(\mZ_{p+q-n}(X)),
$$
that is,
$$L_pH_k(X)\otimes L_qH_l(X)\stackrel{\bullet}{\to} L_{p+q-n}H_{k+l-2n}(X).
$$

\section{Correspondences}

In this section we recall basic materials of correspondences and their actions on Lawson homology
(cf. \cite{Friedlander-Mazur}, \cite{Peters}, \cite{Hu-Li}).
For closely related materials  on Chow correspondences we refer to Manin \cite{Manin} and Fulton \cite{Fulton}.

 A \textbf{correspondence} $\Gamma$ from $X$ to
$Y$ is an algebraic cycle (or an equivalent class of cycles depending on the
context) on $X\times Y$.
We denote the group of correspondences of rational equivalence
classes between varieties $X$ and $Y$ by
$$Corr_d(X, Y):={A}_{\dim X+d}(X\times Y).$$

Let $X$, $Y$ be smooth projective varieties and let $\Gamma\in Corr_d(X, Y)$ for  $d\in \Z$. Then for any
element $\alpha\in L_pH_k(X)$, the push-forward morphism is defined
by
$$
\begin{array}{cc}
&\Gamma_*: L_pH_k(X)\rightarrow L_{p+d}H_{k+2d}(Y)\\
& \Gamma_*(\alpha)=p_{2*}(p_1^*\alpha\bullet \Gamma),
\end{array}
$$
where $p_1$ (resp. $p_2$) denotes the projection from $X\times Y$
onto $X$ (resp. $Y$) and  ``$\bullet$" is the intersection product
on the  group  $A(X\times Y)$.

Let $X$, $Y$, $Z$ be smooth projective varieties. The composition of
two correspondences $\Gamma_1\in Corr_{d_1}(X,Y)$ and $\Gamma_2\in
Corr_{d_2}(Y,Z)$ is given by the formula
$$\Gamma_2\circ\Gamma_1=p_{13*}(p_{12}^*\Gamma_1\cdot p_{23}^*\Gamma_2)\in Corr_{d_1+d_2}(X, Z)
$$
where $p_{ij}$, $i,j =1, 2, 3$ are the projection of $X\times Y\times Z$ on the
product of its $i$th and $j$th factors.

\begin{lemma}\label{lemma1}
Let $X, Y, Z$ be smooth projective varieties, $\Gamma_1\in
  Corr_d(X,Y)$ and $\Gamma_2\in Corr_e(Y, Z)$. Then for any $u\in
  L_pH_k(X)$, we have
  $$(\Gamma_2\circ\Gamma_1)_*u=\Gamma_{2*}\Gamma_{1*}u \in L_{p+d+e}H_{k+2d+2e}(Z).$$
\end{lemma}
\begin{proof}
cf. \cite[Prop. 4.2]{Hu-Li}.
\end{proof}

For a projective morphism $f:X_1\to X_2$ the \emph{graph} of $f$ is defined to be the correspondence
$$
\Gamma_f:=(\id_{X_1}, f)_*(X_1)\in A(X_1\times X_2).
$$
\begin{lemma}\label{lemma2}
\begin{enumerate}
\item $(\Gamma_{f})_*(\alpha)=f_*(\alpha)$ for $\alpha\in L_pH_k(X_1)$.
\item  $(^t\Gamma_{f})_*(\beta)=f^*(\beta)$ for $\beta\in L_qH_l(X_2)$.
\end{enumerate}
\end{lemma}
\begin{proof}
From the projection formula (cf. \cite[Lemma 11. c)]{Peters}), we have
$$
\begin{array}{ccl}
(\Gamma_{f})_*(\alpha)
&=&p_{2*}\big((\id_{X_1},f)_*(X_1)\bullet p_1^*\alpha\big)\\
&=&p_{2*}(\id_{X_1},f)_*\big((\id_{X_1},f)^* p_1^*\alpha\big)\\
&=&f_*(\alpha),
\end{array}
$$
since $p_1\circ (\id_{X_1}, f)=\id_{X_1}$ and $p_2\circ (\id_{X_1}, f)=f$.

The proof of (2) is similar.
\end{proof}

\begin{lemma}\label{lemma3}
Let $f_i:Y_i\to X_i$, for $i=1,2$ be projective morphisms of smooth projective varieties. Then
\begin{enumerate}
\item $(f_1\times f_2)^*Z=^t\Gamma_{f_2}\circ Z\circ \Gamma_{f_1}$ for all $Z\in A(X_1\times X_2)$;
\item $(f_1\times f_2)_*\tilde{Z}=\Gamma_{f_2}\circ \tilde{Z}\circ ^t\Gamma_{f_1}$ for all $\tilde{Z}\in A(Y_1\times Y_2)$.
\end{enumerate}
\end{lemma}
\begin{proof}
We give a proof for 1), the proof of 2) is similar. To prove 1) it is enough to show that  $(f_1\times id)^*Z=Z\circ \Gamma_{f_1}$ and
$(id\times f_2)^*Z=^t\Gamma_{f_2}\circ Z$. Denote by $q_1$ the projection $Y_1\times X_2\to Y_1$ and $p_{ij}$ the projections of $Y_1\times X_1\times X_2$, e.g. $p_{12}:Y_1\times X_1\times X_2\to Y_1\times X_1$ is the projection onto the  Cartesian product first two varieties. Then
by applying the base change formula for Lawson homology (cf. \cite[Lemma 11 a)]{Peters}) to $(id_{Y_1}, f_1)\circ q_1=p_{12}\circ((id_{Y_1},f_1)\times id_{X_2})$ and the projection formula (cf. \cite[Lemma 11 c)]{Peters}), we have
$$
\begin{array}{ccl}
Z\circ \Gamma_{f_1}&=&p_{13*}(p_{23}^*Z\cdot p_{12}^*(id_{Y_1},f_1)_*(Y_1))\\
&=&p_{13*}(p_{23}^*Z\cdot ((id_{Y_1},f_1)\times id_{X_2})_*q_1^*(Y_1))\\
&=&p_{13*}((id_{Y_1},f_1)\times id_{X_2})_*((id_{Y_1},f_1)\times id_{X_2})^*p_{23}^*Z\cdot q_1^*(Y_1))\\
&=&p_{13*}((id_{Y_1},f_1)\times id_{X_2})_*((id_{Y_1},f_1)\times id_{X_2})^*p_{23}^*Z)\\
&=&(f_1\times id_{X_2})^*Z,
\end{array}
$$
where we used $p_{13}\circ ((id_{Y_1},f_1)\times id_{X_2})=id_{Y_1\times X_2}$  and $p_{23}\circ ((id_{Y_1},f_1)\times id_{X_2})=f_1\times id_{X_2}$. From this we get
$$
(id_{X_1}\times f_2)^*Z= {^t}((f_2\times id_{X_1})^* {^t}Z)={^t}( ^tZ\circ \Gamma_{f_2})={^t}\Gamma_{f_2}\circ Z.
$$
\end{proof}

From Lemma \ref{lemma1}-\ref{lemma3},  we have
\begin{equation}\label{eq02}
(f_1\times f_2)^*Z_*(\alpha)=f_2^*(Z_*f_{1*}(\alpha)),  \forall \alpha\in L_*H_*(Y_1)
\end{equation}
and
\begin{equation}\label{eq03}
(f_1\times f_2)_*Z_*(\beta)=f_{2*}(Z_* f_1^*(\beta)), \forall \beta\in L_*H_*(X_1).
\end{equation}

\section{Fourier-Mukai transform}
Let $X$ be an abelian variety of dimension $n$ and let $\widehat{X}$ be the dual abelian variety of $X$, i.e., $\widehat{X}=Pic^0(X)$. Consider the Poincar\'{e} bundle $\mP=\mP_X\in Pic(X\times\widehat{X})=\Ch^1(X\times \widehat{X})$.
The correspondence
$$
e^{\mP}:=\sum_{i\geq0}\frac{1}{i!}\mP^i\in \Ch_*(X)_{\Q}
$$
is well-defined since the sum is finite, where $\mP^i$ denotes the $i$-th intersection product.

The Fourier-Mukai transform on Chow group $\Ch(X)_{\Q}$ with rational coefficients is defined to be the homomorphism of groups
$F=F_X:\Ch(X)_{\Q}\to \Ch(\widehat{X})_{\Q} $, $\alpha\mapsto p_{2*}(e^{\mP}\cdot p_1^*\alpha)$.

Similarly, the \textbf{Fourier-Mukai transform} on Lawson homology $L_*H_*(X)_{\Q}$ with rational coefficients  is defined to be the homomorphism of groups
$F=F_X:L_*H_*(X)_{\Q}\to L_*H_*(\widehat{X})_{\Q} $, $\alpha\mapsto p_{2*}(e^{\mP}\cdot p_1^*\alpha)$.
In particular,  the {Fourier-Mukai transform} on  homology $H_*(X,\Q)$ with rational coefficients  is defined to be the homomorphism of groups
$F=F_X: H_*(X,\Q)\to H_*(\widehat{X},\Q) $, $\alpha\mapsto p_{2*}(e^{\mP}\cdot p_1^*\alpha)$.

\begin{theorem}[Inversion Theorem]\label{Th4} Let $X$ be an abelian variety of dimension $n$. Then we have
 $$F_{\widehat{X}}\circ F_X=(-1)^d (-1)^*_X:L_*H_*(X)_{\Q}\to L_*H_*(X)_{\Q},$$
 where $-1:X\to X$ is the multiplication by $-1$ on $X$.
\end{theorem}
\begin{proof}
By definition, we need to show that
$$
e^{\mP_{\widehat{X}}}\circ e^{\mP_X}=(-1)^d\Gamma_{-1}\in L_*H_*(X)_{\Q}.
$$

Since (cf. \cite{Mukai})
$$
e^{\mP_{\widehat{X}}}\circ e^{\mP_X}=(-1)^d\Gamma_{-1}\in \Ch_*(X\times X)_{\Q},
$$
we get  $$e^{\mP_{\widehat{X}}}\circ e^{\mP_X}=(-1)^d\Gamma_{-1}\in A_*(X\times X)_{\Q}$$

Since by construction the action of correspondence $e^{\mP_{\widehat{X}}}\circ e^{\mP_X}$ on Lawson homology depends only on its class in  $A_*(X\times X)$ (cf. \cite{Peters}), it implies that $e^{\mP_{\widehat{X}}}\circ e^{\mP_X}=(-1)^d\Gamma_{-1}:L_*H_*(X)_{\Q}\to L_*H_*(X)_{\Q}.$


\end{proof}

\begin{proposition}\label{prop5}
 Let $f:Y\to X$  be an isogeny of abelian varieties. Then for  all $\alpha\in L_*H_*(Y)_{\Q}$ and $\beta\in L_*H_*(X)_{\Q}$
\begin{enumerate}
 \item $F_X f_*(\alpha)=\hat{f^*}F_Y(\alpha)$;
\item $F_Yf^*(\beta)=\hat{f_*} F_X(\beta)$.
\end{enumerate}

\end{proposition}
\begin{proof}
\begin{enumerate}
\item The universal property of the Poincar\'{e} bundle implies that
$$(f\times \id_{\widehat{X}})\mP_X=(\id_Y\times \hat{f})^*\mP_Y.$$

By Equation (\ref{eq02}), we get
$$
\begin{array}{ccl}
F_Xf_*(\alpha)&=& (e^{\mP_X})_*f_*\alpha\\
&=&(f\times id_{\widehat{X}})^* (e^{\mP_X})_*\alpha\\
&=&(id_Y\times \hat{f})^* (e^{\mP_Y})_*\alpha\\
&=&\hat{f}^* (e^{\mP_Y})_*\alpha\\
&=&\hat{f^*}F_Y(\alpha).
\end{array}
$$

\item By applying (1) to $\hat{f}:\widehat{X}\to \widehat{Y}$ we get
$$
\begin{array}{ccll}
F_Yf^*&=&(-1)^g(-1)^*_{\widehat{Y}}F_Yf^*F_{\widehat{X}}F_X & \hbox{(By  Theorem \ref{Th4})}\\
&=&(-1)^g(-1)^*_{\widehat{Y}}F_YF_{\widehat{Y}}\hat{f}^*F_X & \hbox{(By Part (1))}\\
&=&\hat{f}^*F_X. & \hbox{(By Theorem \ref{Th4})}\\
\end{array}
$$
\end{enumerate}
\end{proof}

Now we show that the Fourier-Mukai transform $F:L_*H_*(X)_{\Q}\to L_*H_*(X)_{\Q}$ on Lawson homology groups is compatible to
that on the  singular homology with rational coefficients. That is, we have the following result.
\begin{proposition}\label{prop6}
There is a commutative diagram
$$
\xymatrix{
&L_pH_k(X)_{\Q}\ar[r]^{F_X}\ar[d]^{\Phi_{p,k,\Q}}&L_pH_k(\widehat{X})_{\Q}\ar[d]^{\Phi_{p,k,\Q}}\\
&H_k(X,\Q) \ar[r]^{F_X} &H_k(\widehat{X},\Q)
}
$$
\begin{proof}
It follows from the definitions of the Fourier-Mukai transform on Lawson homology and the singular homology.
\end{proof}
\end{proposition}

\section{Pontryagin Product}

In this section,  $X$ denotes an abelian variety of dimension $n$. Let $\mu:X\times X\to X$ denote the sum map $\mu(z,z')=z+z'$. The morphism
$\mu$ induces a continuous map  $\mu_*:\mZ_p(X\times X)\to \mZ_p(X)$ between the space of algebraic cycles.

For $Z=\sum n_iV_i\in \mZ_p(X)$ and $Z'=\sum m_jW_j\in \mZ_q(X)$, we set $Z\times Z'=\sum n_im_jV_i\times W_j\in
\mZ_{p+q}(X\times X)$. So we get a bilinear continuous map $\times:\mZ_p(X)\times \mZ_q(X)\to \mZ_{p+q}(X\times X)$.

Therefore we get a continuous composed map $\mu_*\circ \times:\mZ_p(X)\times \mZ_q(X)\to \mZ_{p+q}(X)$.
We choose the ``empty cycle" $\emptyset_p$ (resp. $\emptyset_q$, $\emptyset_{p+q}$) as the base point
 in $\mZ_p(X)$ (resp. $\mZ_q(X)$, $\mZ_{p+q}(X)$) so that each of $\mZ_p(X)$,  $\mZ_q(X)$ and $\mZ_{p+q}(X)$ is a point topological abelian group.
Note that we have both $\mu_*\circ \times (Z\times \emptyset_q)=\emptyset_{p+q}$ and  $\mu_*\circ \times (\emptyset_p\times Z')=\emptyset_{p+q}$
This implies that the map $\times$ factors through $\mZ_p(X)\wedge \mZ_q(X)$, i.e., there is a commutative diagram of continuous maps
$$
\xymatrix{\mZ_p(X)\times \mZ_q(X)\ar[rd]^{\times}\ar@{.>}[d]&&\\
\mZ_p(X)\wedge \mZ_q(X)\ar@{.>}[r]&\mZ_{p+q}(X\times X)\ar[r]^-{\mu_*} &\mZ_{p+q}(X).
}
$$

For $\alpha\in L_pH_k(X)$ and $\beta\in L_qH_l(X)$, the \textbf{Pontryagin Product}
$$*: L_pH_k(X)\otimes L_qH_l(X)\to L_{p+q}H_{k+l}(X), ~(\alpha,\beta)\mapsto \alpha*\beta
$$
is defined to be  the image of  the homotopy class of $f\wedge g:S^{k+l-2(p+q)}\to \mZ_{p+q}(X)$ under $\mu_*$, where $f:S^{k-2p}\to \mZ_p(X)$
 (resp. $g:S^{l-2q}\to \mZ_q(X)$) is a representative element of $\alpha$ (resp. $\beta$) and $f\wedge g$ is the composed map
$S^{k+l-2(p+q)}=S^{k-2p}\wedge S^{l-2q}\to \mZ_p(X)\wedge \mZ_q(X)\to \mZ_{p+q}(X\times X)$.
The homotopy class of $f\wedge g$ defines an element in $L_{p+q}H_{k+l}(X\times X)$ and so $\mu_*([f\wedge g])$ gives  us
an element in  $L_{p+q}H_{k+l}(X)$. By comparing  the intersection  of Lawson homology defined by Friedlander and Gabber (\cite{Friedlander-Gabber}),
 this product  $\alpha*\beta:=\mu_*([f\wedge g])$ is equal to $\mu_*(p_1^*\alpha\bullet p_2^*\beta)$.

\begin{lemma}
 The Pontryagin product is bilinear, associative and anti-commutative for the second index on $L_*H_*(X)$.
\end{lemma}
\begin{proof} The bilinear property follows exactly from the above definition of $*$.
Note that $f\wedge g=(-1)^{k+l-2(p+q)}g\wedge f=(-1)^{k+l}g\wedge f$, where $f,g$ are given as above.
 Hence we get the anti-commutativity for the second index.
 The associativity follows from the associativity of $\wedge$.
\end{proof}

\begin{proposition}\label{Prop4}
 For all $\alpha, \beta\in L_*H_*(X)$, we have
\begin{enumerate}
 \item $F(\alpha*\beta)=F(\alpha)\cdot F(\beta)$;

\item $F(\alpha\cdot \beta)=(-1)^n F(\alpha)*F(\beta)$.
\end{enumerate}
\end{proposition}

\begin{proof}
\begin{enumerate}
 \item
 We denote by $q_i$ and $q_{ij}$ the projections of $X\times X\times \widehat{X}$. Note first that
  $(\mu\times 1_{\widehat{X}})^*{\mP_X}=q_{13}^*{\mP_X}\cdot q_{23}^*{\mP_X}$ holds in $\Ch_*(X\times X\times \widehat{X})$
  (cf. \cite[Lemma 14.1.7]{Birkenhake-Lange}) implies that $(\mu\times 1_{\widehat{X}})^*e^{\mP_X}=q_{13}^*e^{\mP_X}\cdot q_{23}^*e^{\mP_X}$
  holds in $\Ch_*(X\times X\times \widehat{X})$ and so in $A(X\times X\times \widehat{X})$. Then
$$
\begin{array}{ccll}
 F(\alpha*\beta)
&=&p_{2*}\big(e^{\mP_X}\cdot p_1^*(\mu(\alpha\times\beta))\big)&\\

&=& p_{2*}(e^{\mP_X}\cdot (\mu\times 1_{\widehat{X}})_*q^*_{12}(\alpha\times\beta))&\\

&=& p_{2*}(e^{\mP_X}\cdot (\mu\times 1_{\widehat{X}})_*q^*_{1}\alpha\cdot q_2^*\beta))&\\

&=& p_{2*}(\mu\times 1_{\widehat{X}})_*\big((\mu\times 1_{\widehat{X}})^*e^{\mP_X}\cdot q^*_{1}\alpha\cdot q_2^*\beta)\big)&\\

&=& p_{2*}(\mu\times 1_{\widehat{X}})_*\big(q_{13}^*e^{\mP_X}\cdot q_{23}^*e^{\mP_X}\cdot q^*_{1}\alpha\cdot q_2^*\beta)\big)&\\

&=& q_{3*}(q_{13}^*e^{\mP_X}\cdot q_1^*\cdot q_{23}^*e^{\mP_X}\cdot q_2^*\beta)&\\
&&\hbox{\quad(since $q_3=\mu\times id_{\widehat{X}}$)}&\\

&=& p_{2*}q_{13*}(q_{13}^*e^{\mP_X}\cdot q_1^*\alpha\cdot q_{23}^*e^{\mP_X}\cdot q_2^*\beta)&\\
&&\hbox{\quad{(since $q_3=p_2\circ q_{13}$)}}&\\

&=& p_{2*}q_{13*}\big(q_{13}^*(e^{\mP_X}\cdot p_1^*\alpha)\cdot q_{23}^*(e^{\mP_X}\cdot p_1^*\beta)\big)\\
&&\hbox{\quad{(since $q_1=p_1\circ q_{13},q_2=p_1\circ q_{23}$)}}&\\

&=& p_{2*}\big(e^{\mP_X}\cdot p_1^*\alpha\cdot q_{13*}q_{23}^*(e^{\mP_X}\cdot p_1^*\beta)\big)\\
&&\hbox{\quad{(since $q_3=p_2\circ q_{13}$ and $q_3=p_2\circ q_{23}$)}}&\\

&=& p_{2*}\big(e^{\mP_X}\cdot p_1^*\alpha\cdot p_{2}^*p_{2*}(e^{\mP_X}\cdot p_1^*\beta)\big)\\

&=& p_{2*}(e^{\mP_X}\cdot p_1^*\alpha)\cdot p_{2*}(e^{\mP_X}\cdot p_1^*\beta))\\

&=& F(\alpha)\cdot F(\beta).\\
\end{array}
$$

\item The statement (2) follows from Part (1) by the Inversion Theorem \ref{Th4}.
\end{enumerate}
\end{proof}

\begin{proposition}
The Pontryagin product is compatible with the natural transformation $\Phi_{p,k}:L_pH_k(X)\to H_k(X)$. More precisely, we have the
following commutative diagram:
$$
\xymatrix{L_pH_k(X)\otimes L_qH_l(X)\ar[r]^-{*}\ar[d]^{\Phi_{p,k}\otimes \Phi_{q,l}}&L_{p+q}H_{k+l}(X)\ar[d]^{\Phi_{p+q,k+l}}\\
H_k(X)\otimes H_l(X)\ar[r]^-{*}& H_{k+l}(X).
}
$$
\end{proposition}

\begin{proof}
Let $\alpha\in L_pH_k(X)$ (resp. $\beta\in L_qH_l(X)$) be represented by the homotopy
class of the map $f:S^{k-2p}\to \mZ_{p}(X)$ (resp. $g:S^{l-2q}\to \mZ_{q}(X)$).
Then by definition we have $\alpha*\beta=\mu_*([f\wedge g])$.

Recall that from the property of $s$-map (cf. \cite[Chapter 6]{Friedlander-Mazur} ), one has the explicitly formulas
$$
\Phi_{p,k}([f])=[f\wedge S^{2p}]; \Phi_{q,l}([g])=[g\wedge S^{2q}].
$$

Hence
$$
\begin{array}{cclr}
\Phi_{p,k}(\alpha)*\Phi_{q,l}(\beta)
&=&\Phi_{p,k}([f])*\Phi_{q,l}([g])&\\
&=&\mu_*[f\wedge S^{2p}\wedge g\wedge S^{2q}]&\\
&=&\mu_*[f\wedge g\wedge S^{2p}\wedge S^{2q}]&\\
&=&\mu_*[f\wedge g\wedge S^{2(p+q)}]&\\
&=&\mu_*\big(\Phi_{p+q,k+l}([f\wedge g])\big)&\\
&=&\Phi_{p+q,k+l}\big(\mu_*([f\wedge g])\big)&\\
&=&\Phi_{p+q,k+l}\big(\alpha*\beta\big).&\\
\end{array}
$$
The penultimate equality holds since $\Phi_{p+q,k+l}:L_{p+q}H_{k+l}(X)\to H_{k+l}(X)$ is a natural transformation
from Lawson homology to the singular homology.  This completes the proof of the commutative diagram.
\end{proof}

\begin{remark}
From the proof of the above proposition, we observe that the Pontryagin  product $*$ is compatible with the $s$-map.
\end{remark}


\section{Decompositions of Lawson homology groups for abelian varieties}
\label{sec6}
Let $X$ be an abelian variety of dimension $n$. For each integer $m$, there is a homomorphism $m_X:X\to X$ defined by $x\mapsto m\cdot x.$
Recall that we have cycle class map $\Phi_{p,k}\otimes{\Q}:L_pH_k(X)_{\Q}\to H_k(X,\Q)$ for all $k\geq 2p\geq 0$.
By considering elements in $H_{k}(X,\Q)\cong H^{k}(X,\Q)$ as the dual of differential forms, it is easy to see that the induced map
$m_{X*}:H_{k}(X,\Q)\to H_{k}(X,\Q)$ is multiplication by $m^k$.

There is an eigenspace decomposition of $L_pH_k(X)_{\Q}$ for each pair of $p,k$ such that $k\geq 2p\geq 0$. Set
$$
L_pH_k(X)_{\Q}^s:=\{ \alpha\in L_pH_k(X)_{\Q}| m_{X*}\alpha=m^{k+s}\alpha, \forall ~m \in\Z\}
$$

\begin{theorem}\label{Th5}
Let $X$ be an abelian variety of dimension $n$. Then we have the following decomposition
 \begin{equation}\label{eq04}
 L_pH_k(X)_{\Q}=\bigoplus^{n-[\frac{k+1}{2}]}_{s=p-k}L_pH_k(X)_{\Q}^s,
 \end{equation}
 where $[a]$ denotes the largest integer less than or equal to $a$.
\end{theorem}

We have a direct corollary from Theorem \ref{Th5}.
\begin{corollary}\label{cor10}
Let $X$ be an abelian variety of dimension $n$. Then
$L_pH_k(X)_{\Q}^s=0$ for $s> n-[\frac{k+1}{2}]$ or $s<p-k$.
\end{corollary}

\begin{lemma}\label{Lemma6}
 Suppose $\alpha\in L_pH_k(X)_{\Q}$ and
 $$F(\alpha)=\sum_{q=p-[\frac{k}{2}]}^n\beta_q$$
 with $\beta_q\in L_qH_l(\widehat{X})_{\Q}$, where $l=k+2(q-p)$. Then
for all $m\in \Z$, we have
$$
m^*_{\widehat{X}}\beta_q=m^{n-q+p}\beta_q.
$$
\end{lemma}

\begin{proof}
 By the definition of $F$ we have
$$
\beta_q=\frac{1}{(n-q+p)!}p_{2*}(\mP^{(n-q+p)}\cdot p_1^*\alpha)\in L_qH_l(\widehat{X})_{\Q}.
$$
Hence using flat base change with $m_{\widehat{X}}\circ p_2=p_2\circ(1_X\times m_{\widehat{X}})$ (cf. \cite[\S3]{Friedlander-Gabber}) and the fact that $(1_X\times m_{\widehat{X}})^*  \mP=m\mP$, we get
$$
\begin{array}{lcl}
m_{\widehat{X}}^*\beta_q
&=&\frac{1}{(n-q+p)!}m_{\widehat{X}}^* p_{2*}(\mP^{(n-q+p)}\cdot p_1^*\alpha)\\
&=& \frac{1}{(n-q+p)!}p_{2*}((1_X\times m_{\widehat{X}})^*\mP^{(n-q+p)}\cdot p_1^*\alpha)\\
&=& \frac{m^{n-q+p}}{(n-q+p)!} p_{2*}(\mP^{(n-q+p)}\cdot p_1^*\alpha)\\
&=& m^{n-q+p}\beta_q.
\end{array}
$$

\end{proof}

\begin{proposition}\label{Prop7}
 For $\alpha\in L_pH_k(X)_{\Q}$ and $m\in \Z-\{-1,0,1\}$ the following statements are equivalent:
\begin{enumerate}
 \item $\alpha\in L_pH_k(X)_{\Q}^s$,

\item $m_{X*}\alpha=m^{k+s}\alpha$,
\item $m_X^*\alpha=m^{2n-k-s}\alpha$,
\item $F(\alpha)\in L_{n-k+p-s}H_{2n-2s-k}(\widehat{X})_{\Q}$,
\item $F(\alpha)\in L_{n-k+p-s}H_{2n-2s-k}(\widehat{X})_{\Q}^s$.
\end{enumerate}
\end{proposition}
\begin{proof} We show this in the following way: $(1)\Leftrightarrow (2)\Leftrightarrow (3)\Rightarrow (4)\Rightarrow (5)\Rightarrow (3)$.
 \item {\bf $(1)\Leftrightarrow (2)$} is from the definition of $L_pH_k(X)_{\Q}^s$.

\item  $(2)\Rightarrow (3)$. Note first that $m^*_X\alpha\in L_pH_k(X)_{\Q}$. Suppose $m^*_X\alpha\in L_pH_k(X)^{s'}_{\Q}$
 and from the definition we get $m_{X*}( m_X^*\alpha)=m^{k+s'}m_X^*\alpha$. Since $m_{X*}( m_X^*\alpha)=(\deg m_X)\alpha=m^{2n}\alpha$,
  we obtain that $ m_X^*\alpha=m^{2n-k-s'}\alpha\in L_pH_k(X)_{\Q}^s$. This implies that $s=s'$.

\item $(3)\Rightarrow (2)$. From $m^{2n}\alpha=\deg(m_X) \alpha=m_{X*}( m_X^*\alpha)=m_{X*}( m^{2n-k-s}\alpha )=m^{2n-k-s}m_{X*}\alpha$ we get $m_{X*}\alpha=m^{k+s}\alpha$.

\item $(3)\Rightarrow (4)$. We  write $F_X(\alpha)=\sum_q\beta_q$ with  $\beta_q\in L_qH_{k+2(q-p)}(\widehat{X})_{\Q}$. Then
$$
\begin{array}{ccl}
\sum_{q=p-[\frac{k}{2}]}^n\beta_q&=&F_X(\alpha)\\
&=&\frac{1}{m^{k+s}}F(m_{X*}\alpha)\\
&=&\frac{1}{m^{k+s}}m^*_{\widehat{X}}F(\alpha)\\
&=&\frac{1}{m^{k+s}}\sum_{q=0}^n m^*_{\widehat{X}}\beta_q\\
&=&\frac{1}{m^{k+s}}\sum_{q=0}^n m^{n-q+p}\beta_q\\
&=& m^{n-k-q+p-s}\beta_q.
\end{array}
$$
Comparing coefficients this implies that $$F(\alpha)=\beta_{n-k+p-s}\in L_{n-k+p-s}H_{2n-k-2s}(\widehat{X})_{\Q}.$$

\item $(4)\Rightarrow (5)$. By Lemma \ref{Lemma6}, we have $F(\alpha)\in L_{n-k+p-s}H_{2n-2s-k}(\widehat{X})_{\Q}^s$.

\item $(5)\Rightarrow (3)$. For every $m\in\Z$, we have
$$
\begin{array}{cclr}
 m_X^*\alpha&=&m_X^*(-1)^n(-1)^*_X F_{\widehat{X}}F_{X}\alpha&\hbox{(by Theorem \ref{Th4})}\\
&=&(-1)^n(-1)^*_X F_{\widehat{X}} m_{\widehat{X}*}F_{X}\alpha& \hbox{(by Proposition \ref{Prop4})}\\
&=&(-1)^n(-1)^*_X F_{\widehat{X}} m_{\widehat{X}*} \beta_{n-k+p-s}&\hbox{(by Statement (5))}\\
&=&m^{-k-s}(-1)^n(-1)^*_X F_{\widehat{X}} m_{\widehat{X}*}m_{\widehat{X}}^*\beta_{n-k+p-s}& \hbox{(by Lemma \ref{Lemma6})}\\
&=&m^{2n-k-s}(-1)^n(-1)^*_X F_{\widehat{X}}  \beta_{n-k+p-s}&\\
&=&m^{2n-k-s}(-1)^n(-1)^*_X F_{\widehat{X}} F_X(\alpha) &\\
&=&m^{2n-k-s}\alpha.&
\end{array}
$$

\end{proof}

\begin{proof}[Proof of Theorem \ref{Th5}]
Suppose $\alpha\in L_pH_k(X)_{\Q}$ and write $F_X(\alpha)=\sum\beta_q$ with  $\beta_q\in L_qH_{k+2(q-p)}(\widehat{X})_{\Q}$.
 By Lemma \ref{Lemma6} we have $\beta_q\in L_qH_{k+2(q-p)}(\widehat{X})_{\Q}^{n-k-q+p}$.  By applying Proposition \ref{Prop7} to
  $\beta_q$, we get $$F_{\widehat{X}}(\beta_q)\in  L_pH_k(X)_{\Q}^{n-k-q+p}.$$

Now by Theorem \ref{Th4},
$$
\begin{array}{ccl}
\alpha&=&(-1)^n(-1)^*F_{\widehat{X}}\circ F_X(\alpha)\\
&=&(-1)^n(-1)^*\sum^n_{q=p-[\frac{k}{2}]}\beta_q\in \bigoplus^n_{q=p-[\frac{k}{2}]} L_pH_{k}(X)^{n-k-q+p}_{\Q}.
\end{array}
$$
This implies  the assertion since $n-k+[\frac{k}{2}]=n-[\frac{k+1}{2}]$.
\end{proof}

The decomposition of Equation (\ref{eq04}) is compatible with many natural maps.
\begin{proposition} Let $f:X\to Y$ be a group homomorphism between abelian varieties. Then induced map $f_*$ preserves
the decomposition of Lawson homology groups with rational coefficients, i.e.,
$$
f_*(L_pH_k(X)_{\Q}^s)\subseteq L_pH_k(Y)_{\Q}^s, ~\forall s\in\Z.
$$
\end{proposition}
\begin{proof}
Since $f:X\to Y$ is a group homomorphism, one has $f(m\cdot x)=m\cdot f(x)$ and so $m_Y\circ f= f\circ m_X$. Hence
we have $m_{Y*}\circ f_*=f_*\circ m_{X*}$. This completes the proof of the proposition.
\end{proof}

\begin{proposition} \label{prop12}
The decomposition in Equation (\ref{eq04}) is  compatible with the cycle class map $\Phi_{p,k,\Q}$.
\end{proposition}
\begin{proof}
It follows from the fact that the multiplication $m_X$ by $m$ on $X$ commutes with the Fourier-Mukai transform $F_X$ (cf. Proposition \ref{prop5})
and the cycle class map $\Phi_{p,k,\Q}$, the latter is the general fact that the cycle class map $\Phi_{p,k,\Q}$ is a natural transformation between
the Lawson homolgy (cf. \cite{Friedlander-Mazur}, \cite[Ch. IV]{Lawson2}).
\end{proof}

Set
$$
L_pH_{k}(X)_{hom,\Q}^s:=
\{ \alpha\in L_pH_k(X)_{hom,\Q}| m_{X*}\alpha=m^{k+s}\alpha, \forall ~m \in\Z\}.
$$

From Proposition \ref{prop12}, we have
\begin{equation}
L_pH_k(X)_{hom,\Q}^s=
L_pH_k(X)_{\Q}^s, \quad s\neq0
\end{equation}
and
\begin{equation}
L_pH_k(X)_{hom,\Q}^0=L_pH_k(X)_{\Q}^0\cap L_pH_k(X)_{hom,\Q}.
\end{equation}

On the image of the natural transform $\Phi_{p,k}\otimes{\Q}:L_pH_k(X)_{\Q}\to H_k(X)_{\Q}$, we have the following result.
\begin{corollary}\label{cor6.7}
Let $X$ be an abelian variety of dimension $n$. Then we have
$$
T_pH_k(X)_{\Q}\cong T_{n+p-k}H_{2n-k}(\widehat{X})_{\Q}
$$
\end{corollary}
\begin{proof}
Note that we have the following commutative diagram
$$
\xymatrix{L_pH_k(X)_{\Q}^0\ar[r]^-{F_X}\ar[d]^{\Phi_{p,k}\otimes{\Q}}&L_{n+p-k}H_{2n-k}(\widehat{X})_{\Q}^0\ar[d]^{\Phi_{n+p-k,2n-k}\otimes{\Q}}\\
T_pH_k(X)_{\Q}\ar[r]^-{\bar{F}_X}&T_{n+p-k}H_{2n-k}(\widehat{X})_{\Q},
}
$$
where $\bar{F}_X$ is the restriction of the Fourier-Mukai transform $F_X$ on the rational homology groups of $X$. 
Since $F_X:H_k(X,\Q)\to H_k(\widehat{X},\Q)$ is isomorphism, $\bar{F}_X:T_pH_k(X)_{\Q}\to T_{n+p-k}H_{2n-k}(\widehat{X})_{\Q}$ is injective and so
$$\dim_{\Q}T_pH_k(X)_{\Q}\leq \dim_{\Q}T_{n+p-k}H_{2n-k}(\widehat{X})_{\Q}.$$ 
Since $F_{\widehat{X}}$ is also an isomorphism on the rational homology groups, we obtain
$$\dim_{\Q}T_{n+p-k}H_{2n-k}(\widehat{X})_{\Q}\leq \dim_{\Q}T_pH_k(X)_{\Q}.$$ 
Hence $$\dim_{\Q}T_pH_k(X)_{\Q}= \dim_{\Q}T_{n+p-k}H_{2n-k}(\widehat{X})_{\Q}$$
and so $\bar{F}_{X}:T_pH_k(X)_{\Q}\to T_{n+p-k}H_{2n-k}(\widehat{X})_{\Q}$ is an isomorphism. 
\end{proof}

The motivation of our decomposition follows from that of the Chow group theory.
Recall that  there is also an eigenspace decomposition for $\Ch_p(X)$ of for every $p$, due to Beauville \cite{Beauville2}. If we set
$$
\Ch_p(X)_{\Q}^s:=\{\alpha\in \Ch_p(X)_{\Q}|m_{X*}\alpha=m^{2p+s}\alpha ~\hbox{for all  $n\in \Z$} \},
$$
then there is a direct sum decomposition for the Chow group of $X$ with rational coefficients
$$
\Ch_p(X)_{\Q}=\bigoplus_{s=-p}^{n-p}\Ch_p(X)_{\Q}^s.
$$

Beauville conjectures that $\Ch_p(X)_{\Q}^s=0$ for $s<0$ (\cite{Beauville1,Beauville2}).
Similarly, we have the following analogue of  Beauville's  conjecture for the Lawson homology of abelian varieties.
\begin{conjecture}\label{conj}
For an abelian variety $X$, one has
$L_pH_k(X)_{\Q}^s=0$ for $s<0$.
\end{conjecture}

From Theorem \ref{Th4} and Proposition \ref{Prop7}, we have
\begin{corollary}\label{Cor14}
Let $X$ be an abelian variety of dimension $n$. The Fourier-Mukai transformation $F_X$ induces an isomorphism
$$L_pH_k(X)_{\Q}^{s}\cong L_{n-k+p-s}H_{2n-2s-k}(\widehat{X})_{\Q}^{s}
$$
for all integer $s$.
\end{corollary}
\begin{proof}
To see the injectivity of $F_X$, let $\alpha\in L_pH_k(X)_{\Q}^{s}$ such that
\begin{equation}\label{eqn007}
F_X(\alpha)=0\in L_{n-k+p-s}H_{2n-2s-k}(\widehat{X})_{\Q}^s.
\end{equation}

Now we apply $F_{\widehat{X}}$ to both sides of Equation (\ref{eqn007}), we get
\begin{equation}\label{eqn008}
F_{\widehat{X}}\circ F_X(\alpha)= F_{\widehat{X}}(0).
\end{equation}
The right side of Equation (\ref{eqn008}) is obviously equal to zero. By Theorem \ref{Th4},
the left side of equation (\ref{eqn008}) is $(-1)^n(-1)^*_X(\alpha)$. Since $(-1)^*_X$ is an isomorphism, we get $\alpha=0$.

To see the surjectivity of $F_X$, for  $\beta\in  L_{n-k+p-s}H_{2n-2s-k}(\widehat{X})_{\Q}^{s}$, we have
$F_{\widehat{X}}(\beta)\in L_pH_k(X)_{\Q}^{s}$ by Proposition \ref{Prop7}. Set $\alpha:=(-1)^n(-1)^*_{\widehat{X}}$. Then
$F_X(\alpha)=\beta$ by applying Theorem \ref{Th4} to $\widehat{X}$.
\end{proof}

From the explanation in the beginning of this section,  if we define $$
H_k(X)_{\Q}^s:=\{ \alpha\in H_k(X)_{\Q}| m_{X*}\alpha=m^{k+s}\alpha, \forall ~m \in\Z\},
$$
then by applying $F_X$ on singular homology with rational coefficients, one gets
\begin{equation}\label{eq09}
H_k(X)_{\Q}^s=
\left\{
\begin{array}{lcc}
&H_k(X)_{\Q}, &s=0;\\
&0,&s\neq0.
\end{array}
\right.
\end{equation}

Now we will check the situation of the conjecture for low dimensional abelian varieties.
For those $X$ of $\dim X\leq 2$, Friedlander's result \cite[Th.4.6]{Friedlander1} and the Dold-Thom theorem
imply that Conjecture \ref{conj} holds.

\begin{example}
Conjecture \ref{conj} holds for abelian variety $X$ of dimension $3$ except possibly for $p=1, k\geq 4$ and $s=3-k$.
\end{example}
\begin{proof}
From the above discussion, we only  need to consider abelian varieties $X$ of dimension three. That is, we need to show $L_pH_k(X)^s_{\Q}=0$
except for $p=1, k\geq 4$ and $s=3-k$.

There are different cases according to $p$ and $k$. It is trivial when $p<0$ or $p>3$.
\begin{enumerate}
\item $p=0$.  In this case, one has the Dold-Thom isomorphism $L_0H_k(X)\cong H_k(X)$ and therefore by Equation (\ref{eq09})
we have
$$ L_0H_k(X)^s_{\Q}\cong H_k(X)_{\Q}^s=
\left\{
\begin{array}{lcc}
&H_k(X)_{\Q}, &s=0;\\
&0,&s\neq0.
\end{array}
\right.
$$
So we have $L_0H_k(X)_{\Q}^s=0$ for $s<0$ for all $k\geq 0$.

\item $p=2$. In this case, we obtain from Friedlander's theorem (cf. \cite[Th.4.6]{Friedlander1}) that
$$ L_2H_k(X)^s_{\Q}\cong H_k(X)_{\Q}^s=
\left\{
\begin{array}{lcc}
&H_k(X)_{\Q}, &s=0;\\
&0,&s\neq0.
\end{array}
\right.
$$
for $k\geq 5$. For $k=4$, $L_2H_k(X)^s_{\Q}\subseteq H_k(X)_{\Q}^s=0$ for $s<0$ by Equation (\ref{eq09}). So we have $L_2H_k(X)_{\Q}^s=0$ for $s<0$ for all $k\geq 4$.

\item $p=3$. In this case, by definition, the only nontrivial  $L_3H_k(X)$ occurs in the case that $k=6$. When $k=6$,
we have $L_3H_6(X)_{\Q}\cong H_6(X)_{\Q}$. So  $L_3H_6(X)_{\Q}^s\cong H_6(X)_{\Q}^s=0$ for $s<0$ by Equation (\ref{eq09}).

\item $p=1$.
\begin{enumerate}
\item $k=2$.
By applying  Corollary \ref{Cor14} to $X$, we
obtain that
$$
L_1H_2(X)^s_{\Q}\cong L_{2-s}H_{4-2s}(\widehat{X})_{\Q}^s.
$$
So if $s<0$, then $2-s>2$ and hence  $L_{2-s}H_{4-2s}(\widehat{X})_{\Q}^s=0$ by applying the above cases to $\widehat{X}$.

\item $k=3$. Again by applying  Corollary \ref{Cor14} to $X$, we get
$$
L_1H_3(X)^s_{\Q}\cong L_{1-s}H_{3-2s}(\widehat{X})_{\Q}^s.
$$
So if $s<0$, then $1-s\geq 2$ and hence  $L_{1-s}H_{3-2s}(\widehat{X})_{\Q}^s=0$ by the applying the above cases to $\widehat{X}$.

\item $k\geq 4$.  By applying  Corollary \ref{Cor14} to $X$, we get
$$
\begin{array}{ccl}
L_1H_k(X)^s_{\Q}&\cong& L_{4-k-s}H_{6-k-2s}(\widehat{X})_{\Q}^s\\
&=&0,\hbox{if $s\neq 3-k$.}
\end{array}
$$
\end{enumerate}
\end{enumerate}
\end{proof}

From the discuss above, we see that the mysterious part of $L_1H_k(X)_{\Q}$ for an abelian threefold $X$ is 
$L_1H_k(X)_{\Q}^{3-k}=\{\alpha\in L_1H_k(X)_{\Q}|m_{X*}\alpha=m^3\alpha \}$. For $k=2$, it is the Griffiths group with rational coefficients
of $1$-cycles on $X$; For $k=3$, conjecturally it is $T_1H_3(X,\Q)$; For $k\geq 4$, conjecturally it is zero.

Now we describe the relation between of Conjecture \ref{conj} and a weak version conjecture by Suslin
(see the appendix below for the  statement). In the usual (or strong) version of the Suslin conjecture,
 varieties are only required to be quasi-projective. Moreover, the strong version Suslin conjecture  also includes
the statement that the cycle class map $\Phi_{p,k}$ from Lawson homology to the singular homology
is injective for $k=n+p-1$.

\begin{proposition}\label{th19}
Conjecture \ref{conj} is equivalent to the (weak version) Suslin conjecture for Lawson homology with rational coefficients on abelian varieties.
\end{proposition}
\begin{proof}
Let $X$ be an abelian variety of dimension $n$.  First we assume Conjecture \ref{conj}, i.e. $$L_pH_k(X)^s_{\Q}=0$$ for all $k\geq 2p$ and $s<0$.
Then for $k\geq n+p$, we obtain that
\begin{equation}\label{eq10}
\begin{array}{cclr}
 L_pH_k(X)_{\Q}
 &=&\bigoplus^{n-[\frac{k+1}{2}]}_{s=p-k}L_pH_k(X)_{\Q}^s& \hbox{(by Theorem \ref{Th5})}\\
 &=&\bigoplus^{n-[\frac{k+1}{2}]}_{s=0}L_pH_k(X)_{\Q}^s&                 \hbox{(by Conjecture \ref{conj})}\\
 &\cong & \bigoplus^{n-[\frac{k+1}{2}]}_{s=0}L_{n-k+p-s}H_{2n-2s-k}(\widehat{X})_{\Q}^{s}& \hbox{(by Corollary \ref{Cor14})}\\
  &\cong & \bigoplus^{n-[\frac{k+1}{2}]}_{s=0}H_{2n-2s-k}(\widehat{X})_{\Q}^{s}& \hbox{(since $n-k+p-s\leq 0$)}\\
 &= &H_{2n-k}(\widehat{X})_{\Q}^{0}& \hbox{(by Equation (\ref{eq09}))}\\
  &= &H_{2n-k}(\widehat{X})_{\Q}& \hbox{(by Equation (\ref{eq09}))}\\
    &\cong &H_{k}(X)_{\Q}& \hbox{(by Fourier Inversion)}\\
\end{array}
\end{equation}
and this is exactly the Suslin conjecture for $X$ with rational coefficients for $k\geq n+p$.

Conversely, we assume that the Suslin conjecture holds for $X$. Then by Equation (\ref{eq10}), we have $L_pH_k(X)^s_{\Q}=0$ for  $p-k\leq s<0$.
This together with Corollary \ref{cor10} implies that $L_pH_k(X)^s_{\Q}=0$ for all $s<0$.
\end{proof}

\begin{remark}
The weak version Suslin conjecture relates to  a Hard Lefschetz type conjecture  for Lawson homology (cf. \cite{Friedlander-Mazur}, \cite{Xu}).
Moreover, by applying an action of $SL(2,\Z)$ on the space of algebraic cycles modulo the algebraic equivalence, as observed by Beauville, 
Xu gives further direct sum decomposition of $L_pH_k(X)_{\Q}^s$ in terms of primitive elements. 
\end{remark}

The following question was asked by Friedlander and Lawson in the terminology  of their morphic cohomology (for definition,
see section \ref{section7} below).
By  Friedlander-Lawson's duality theorem (cf. \cite{Friedlander-Lawson2}) between the Morphic cohomology and Lawson homology,
an equivalent version in terms of Lawson homology is given in the following.

\begin{question}[Compare with Question 9.7 in {\cite{Friedlander-Lawson}}]\label{ques6.11}
For a smooth projective variety $X$ of dimension $n$, is $\Phi_{p,k}:L_pH_k(X)\to H_k(X)$ surjective for $k\geq n+p$?
\end{question}

Obviously, the Suslin conjecture says Friedlander-Lawson's answer to the above question is yes.
The answer is also yes to the question for abelian varieties, as given  in the morphic cohomology version
 (cf. \cite[Cor. 9.5]{Friedlander-Lawson2}). The affirmative answer to Question \ref{ques6.11} can be obtained by using the
 Friedlander-Lawson's duality theorem.
Now we give an alternative proof of the surjectivity of
$\Phi_{p,k}$, without  using the Friedlander-Lawson's duality theorem.

\begin{proposition}[Friedlander-Lawson]\label{Prop6.12}
For an abelian variety $X$ of dimension $n$, the cycle class map $\Phi_{p,k}:L_pH_k(X)\to H_k(X)$ surjective for $k\geq n+p$.
\end{proposition}
\begin{proof}
From the proof of Proposition \ref{th19}, we see that for any  abelian variety $X$ of dimension $n$,
$\Phi_{p,k}:L_pH_k(X)_{\Q}\to H_k(X)_{\Q}$ is surjective for all $k\geq n+p$.  The details are given as follows:
Since $k\geq n+p$, we have
\begin{equation}\label{eq11}
\begin{array}{cclr}
 L_pH_k(X)_{\Q}
 &=&\bigoplus^{n-[\frac{k+1}{2}]}_{s=p-k}L_pH_k(X)_{\Q}^s& \hbox{(by Theorem \ref{Th5})}\\
 &\supseteq &\bigoplus^{n-[\frac{k+1}{2}]}_{s=0}L_pH_k(X)_{\Q}^s&               \\
 &\cong & \bigoplus^{n-[\frac{k+1}{2}]}_{s=0}L_{n-k+p-s}H_{2n-2s-k}(\widehat{X})_{\Q}^{s}& \hbox{(by Corollary \ref{Cor14})}\\
  &\cong & \bigoplus^{n-[\frac{k+1}{2}]}_{s=0}H_{2n-2s-k}(\widehat{X})_{\Q}^{s}& \hbox{(since $n-k+p-s\leq 0$)}\\
 &= &H_{2n-k}(\widehat{X})_{\Q}^{0}& \hbox{(by Equation (\ref{eq09}))}\\
  &= &H_{2n-k}(\widehat{X})_{\Q}& \hbox{(by Equation (\ref{eq09}))}\\
    &\cong &H_{k}(X)_{\Q}& \hbox{(by Fourier Inversion)}\\
\end{array}
\end{equation}
This implies that $\Phi_{p,k}:L_pH_k(X)_{\Q}\to H_k(X)_{\Q}$ is surjective for $k\geq p+n$.  Since the Suslin conjecture with finite coefficients holds,
as from the work of Milnor-Bloch-Kato conjecture (now a theorem by by Voevodsky, Rost and others),
we obtain that  $\Phi_{p,k}:L_pH_k(X)\to H_k(X)$ is also surjective for $k\geq p+n$.
\end{proof}

\begin{remark}
From the above proposition and  a recent result by  Beilinson \cite{Beilinson},
one obtains an alternative proof of the Grothendieck standard conjecture of Lefschetz type for $X$ (see the appendix below for the statement).
The first proof of this Grothendieck standard conjecture of Lefschetz type  for abelian varieties was obtained by Lieberman \cite{Lieberman}.
\end{remark}

\begin{proposition} Let $X$ be an abelian variety of dimension $n$. Suppose the (strong version) Suslin conjecture holds for $X$.
Then for 1-cycles,  we have  $$L_1H_k(X)_{\Q}^0\cong T_1H_k(X)_{\Q}$$ for all $k\geq 2$ and there is a finite filtration on $L_1H_k(X)_{\Q}$ given by

$$
\begin{array}{ccl}
F^0L_1H_k(X)_{\Q}&=&L_1H_k(X)_{\Q},\\
F^1L_1H_k(X)_{\Q}&=&L_1H_k(X)_{hom,\Q},\\
F^2L_1H_k(X)_{\Q}&=&\ker AJ_X,\\
F^jL_1H_k(X)_{\Q}&=&\bigoplus_{j\geq s} L_1H_k(X)_{\Q}^s,\\
F^jL_1H_k(X)_{\Q}&=&0, j>>1,
\end{array}
$$
where $AJ_X$ is the Abel-Jacobi for Lawson homology, as defined by the author in \cite{Hu2}.

\end{proposition}
\begin{proof}From Theorem \ref{Th5}, we have a finite filtration given by
$$
F^jL_pH_k(X)_{\Q}=\bigoplus_{j\geq s} L_pH_k(X)_{\Q}^s, \quad j=p-k,p-k+1,\cdots, n-[(k+1)/2].
$$

Moveover, $F^{p-k}L_pH_k(X)_{\Q}=L_pH_k(X)_{\Q}$ and $F^{j}L_pH_k(X)_{\Q}=0$ for $j> n-[(k+1)/2]$.
By assumption  and Proposition \ref{th19}, we have
$$F^{j}L_pH_k(X)_{\Q}=L_pH_k(X)_{\Q}$$
for all $j\leq 0$.

By the (strong) Suslin conjecture for Lawson homology with rational coefficients, $\Phi_{p,k}:L_pH_k(X)_{\Q}\to H_k(X)_{\Q}$ is injective for $k=n+p-1$.
By applying Theorem \ref{Th5} to $L_pH_k(X)_{\Q}$ for $k=n+p-1$, we get
\begin{equation}\label{eqn12}
\begin{array}{cclr}
L_pH_k(X)_{\Q}
&=&\bigoplus^{n-[\frac{k+1}{2}]}_{s=p-k}L_pH_k(X)_{\Q}^s& \hbox{(by Theorem \ref{Th5})}\\
&=&\bigoplus^{n-[\frac{k+1}{2}]}_{s=0}L_pH_k(X)_{\Q}^s&                 \hbox{(by Proposition \ref{th19})}\\
&\cong & \bigoplus^{n-[\frac{n+p}{2}]}_{s=0}L_{1-s}H_{n-p+1-2s}(\widehat{X})_{\Q}^{s}& \hbox{(by Corollary \ref{Cor14})}\\
&\cong & L_{1}H_{n-p+1}(\widehat{X})_{\Q}^{0}\\
&&\oplus \bigoplus^{n-[\frac{n+p}{2}]}_{s=1}H_{n-p+1-2s}(\widehat{X})_{\Q}^{s}& \hbox{(by Dold-Thom Theorem)}\\
&= &L_1H_{n-p+1}(\widehat{X})_{\Q}^{0}& \hbox{(by Equation (\ref{eq09}))}\\
&\cong &L_pH_{n+p-1}({X})_{\Q}^0& \hbox{(by Fourier Inversion)}\\
\end{array}
\end{equation}
,
Therefore, $L_pH_{n+p-1}({X})_{\Q}^0\cong L_pH_{n+p-1}(X)_{\Q}$. Then the Sulin conjecture implies that
\begin{equation}\label{eqn13}
L_pH_{n+p-1}({X})_{\Q}^0\cong T_pH_{n+p-1}(X,\Q)
\end{equation}
and hence $F^1L_pH_k(X)_{\Q}=L_pH_k(X)_{hom,\Q}$. In particular, one has $F^1L_1H_k(X)_{\Q}=L_1H_k(X)_{hom,\Q}$  for $p=1$.

Under the assumption of the (strong) Suslin conjecture, for $k\geq 2$, we have the following isomorphisms
$$
\begin{array}{cclr}
L_1H_k(X)_{\Q}^0&\cong& L_{n+1-k}H_{2n-k}(\widehat{X})_{\Q}^0&\hbox{(by Corollary \ref{Cor14})}\\
&\cong&T_{n+1-k}H_{2n-k}(\widehat{X},\Q)&\hbox{(by Equation (\ref{eqn13}))}\\
&\cong&T_1H_k(X,\Q). &\hbox{(by Corollary (\ref{cor6.7}))}
\end{array}
$$

Note that the Abel-Jacobi map for the Lawson homology of a smooth
projective variety $X$ is a natural map $$AJ_X: L_pH_k(X)_{hom}\longrightarrow \bigg\{\bigoplus_{{r>
k-2p+1,r+s=k-2p+1}} H^{p+r,p+s}(X)\bigg\}^*\bigg/H_{k+1}(X,\Z).$$

So if $X$ is an abelian variety and $\alpha\in L_pH_k(X)_{\Q}^s\cap L_pH_k(X)_{hom,\Q}$, then
$$AJ_X(m_{X*}\alpha)=m_{X*}(AJ_X(\alpha)),$$
i.e. $m^{k+s}AJ_X(\alpha)=m^{k+1}AJ_X(\alpha)$ since the action of $m_{X*}$ on $H^{k+1}(X,\Q)$ is the
multiplication by $m^{k+1}$. Therefore $AJ_X(\alpha)=0$ unless $\alpha\in L_pH_k(X)_{\Q}^1$.
\end{proof}

\begin{remark}
From the above proposition, it is reasonable to conjecture for an abelian variety
$X$ that $L_pH_k(X)_{\Q}^0\cong T_pH_k(X,\Q)$ for all integers $k\geq 2p$. Then the above proposition
would hold for all $p\geq 0$.
\end{remark}

The direct sum decomposition for the Lawson homology of an abelian variety $X$ in Theorem \ref{Th5} is compatible with the Pontryagin product.

\begin{proposition}Let $X$ be an abelian variety of dimension $n$. Then the following diagram
$$
\xymatrix{L_pH_k(X)^s_{\Q}\otimes L_qH_l(X)^{s'}_{\Q}\ar[r]^-{*}\ar[d]^{i_{p,k}^s\otimes i_{q,l}^{s'}}&L_{p+q}H_{k+l}(X)^{s+s'}_{\Q}\ar[d]^{i_{p+q,k+l}^{s+s'}}\\
L_pH_k(X)_{\Q}\otimes L_qH_l(X)_{\Q}\ar[r]^-{*}&L_{p+q}H_{k+l}(X)_{\Q}\\
}
$$
is commutative, where $i_{p,k}^s:L_pH_k(X)^s_{\Q}\to L_pH_k(X)_{\Q}$ is the inclusion of the summand in the direct sum decomposition
(cf. Equation (\ref{eq04})).
\end{proposition}
\begin{proof}
We need to show that  for $\alpha\in L_pH_k(X)^s_{\Q}$ and $\beta\in L_qH_l(X)^{s'}_{\Q}$, one has $\alpha*\beta\in L_{p+q}H_{k+l}(X)^{s+s'}_{\Q}$.
Since $\alpha\in L_pH_k(X)^s_{\Q}$, we have by definition $m_{X*}(\alpha)=m^{k+s}\alpha$. Similarly, $m_{X*}(\beta)=m^{l+s'}\beta$.
 It is enough to show that
$m_{X*}(\alpha*\beta)=m^{k+l+s+s'}\alpha*\beta$.
Note that by definition the two morphisms $\mu\circ (m_X,m_X)$ and $m_X\circ \mu$ coincide, i.e.,
$\mu\circ (m_X,m_X)=m_X\circ \mu:X\times X\to X$. So we have
the commutative diagram by the induced maps on Lawson homology groups
$$
\xymatrix{L_pH_k(X)_{\Q}\otimes L_qH_l(X)_{\Q}\ar[r]^-{\times}\ar[d]^{{m_{X*}}\otimes m_{X*}}&L_{p+q}H_{k+l}(X\times X)_{\Q}\ar[r]^-{\mu_*}&L_{p+q}H_{k+l}(X)_{\Q}\ar[d]^{m_{X*}}\\
L_pH_k(X)_{\Q}\otimes L_qH_l(X)_{\Q}\ar[r]^-{\times}&L_{p+q}H_{k+l}(X\times X)_{\Q}\ar[r]^-{\mu_*}&L_{p+q}H_{k+l}(X)_{\Q}.
}
$$
So we have
$$\mu*\circ \times \circ (m_{X*}\otimes m_{X*})=m_{X*}\circ \mu_*\circ \times :L_pH_k(X)_{\Q}\otimes L_pH_k(X)_{\Q}\to L_pH_k(X)_{\Q},$$
in other words,
$$
m_{X*}(\alpha*\beta)=m_{X*}\alpha*m_*\beta.
$$

Now since $\alpha\in L_pH_k(X)^s_{\Q}$ and  $\beta\in L_qH_l(X)^{s'}_{\Q}$, we have
$$
\begin{array}{cclr}
m_{X*} (\alpha*\beta)&=&m_{X*}\alpha*m_*\beta&\\
&=& (m^{k+s}\alpha) * (m^{l+s'}\beta)&\\
&=& m^{k+s+l+s'}\alpha*\beta.& \hbox{(by the bilinearity of $*$)}
\end{array}
$$
So by definition we get $\alpha*\beta\in L_{p+q}H_{k+l}(X)^{s+s'}_{\Q}$.
\end{proof}

However, we  cannot  hope the Pontryagin product  would bring more on the Abel-Jacobi map for Lawson homology, since $AJ_X$ is only
non-trivial on $L_pH_k(X)_{\Q}^1$, but the Pontryagin product of  $\alpha\in L_pH_k(X)_{\Q}^1$ and $\beta\in L_qH_l(X)_{\Q}^1$
is in $L_{p+q}H_{k+l}(X)_{\Q}^2$. So the Abel-Jacobi map $AJ_X$ is identically zero on $L_{p+q}H_{k+l}(X)_{\Q}^2$.


\section{Filtrations on Lawson homology} In this section we will define a filtration on Lawson homology, as an analogue of
the conjectural Bloch-Beilinson filtration for Chow groups.
For $X$ a smooth projective variety, then conjecturally, for every integer $k$, there exists a decreasing filtration $\tilde{F}^jL_pH_k(X)_{\Q}$
on the Lawson homology group with rational coefficients, satisfying the following properties:
\begin{enumerate}
\item $\tilde{F}^{j}L_pH_k(X)_{\Q}=0$ for $j>>1$.
\item The filtration is stable under the action of correspondences:

 If $\Gamma\in A_{\dim X+d}(X\times Y)$, then the maps
$$
\Gamma_*:L_pH_k(X)_{\Q}\to L_{p+d}H_{k+2d}(Y)_{\Q}
$$

satisfy
$$
\Gamma_*(\tilde{F}^{j}L_pH_k(X)_{\Q})\subseteq \tilde{F}^{j}L_{p+d}H_{k+2d}(Y)_{\Q}.
$$
\item The induced map
$$
Gr_{\tilde{F}}^j \Gamma_*: Gr_{\tilde{F}}^{j}L_pH_k(X)_{\Q}\to Gr_{\tilde{F}}^{j}L_{p+d}H_{k+2d}(Y)_{\Q}
$$
vanishes if the class $\Gamma$ is zero on $H_{k+j}(X,\Q)$, i.e.,
$$
0=[\Gamma]_*:H_{k+j}(X,\Q)\to H_{k+j+2d}(Y,\Q).
$$
\end{enumerate}

A candidate of such a filtration could be given in the following way by induction: Assume that
we have constructed $\tilde{F}^{j-1}L_qH_l(Y)_{\Q}$ for every $l\geq 2q$ and every smooth projective variety $Y$.
Then we set $$\tilde{F}^{j}L_pH_k(X)_{\Q}:={\rm span}\{Im ~\Gamma_*(\tilde{F}^{j-1}L_{p+r}H_{k+2r}(Y)_{\Q})|\Gamma\in A^{r+\dim X}(Y\times X)_{\Q}\},$$
where $\Gamma$ satisfies the condition that
$$
\Gamma_*: H_{k+2r}(Y,\Q)\to H_{k}(X,\Q)
$$
is zero.

For $X$ an abelian variety, if we set as before
$$
F^jL_pH_k(X)_{\Q}=\bigoplus_{j\geq s} L_pH_k(X)_{\Q}^s, \quad j=p-k,p-k+1,\cdots, n-[(k+1)/2].
$$

As an analogue in Chow theory (\cite{Murre}), it is reasonable to conjecture that the two filtrations for abelian varieties coincide.


\section{Morphic Cohomology}\label{section7}

There is a  cohomological version of Lawson homology, i.e., the Friedlander-Lawson morphic cohomology, is defined to be the homotopy group of algebraic cocycles. The topological group $\mathcal{Z}^q(X)$ of all algebraic cocycles
of codimension-$q$ on $X$ is defined as a homotopy quotient
completion (cf. \cite[Definition 2.8]{Friedlander-Lawson})
$$\mathcal{Z}^q(X):=[\mathfrak{Mor}(X,\mathcal{C}_0(\mathbb{P}^q))/
\mathfrak{Mor}(X,\mathcal{C}_0(\mathbb{P}^{q-1}))]^+=\mathfrak{Mor}(X,\mathcal{Z}_0(\mathbb{C}^q)).$$

The $(2q-k)$-th homotopy group of the space of algebraic cocycles instead of algebraic cycles, is defined to be  \emph{the Friedlander-Lawson morphic
cohomology group} and is denoted by $L^qH^k(X)$.

Let $X$ be an abelian variety. There is also an eigenspace decomposition of $L^qH^k(X)_{\Q}$ for each pair of $q,k$ such that $k\leq 2q$.
$$L^qH^k(X)_{\Q}^s:=\{\alpha\in L^qH^k(X)_{\Q}|m_X^*\alpha=m^{k-s}\alpha \}
$$

\begin{proposition}[Refined Friedlander-Lawson duality for Abelian varieties]\label{prop23}
There is an induced isomorphism
$$
L^qH^k(X)_{\Q}^s\cong L_{n-q}H_{2n-k}(X)_{\Q}^s
$$
from the Friedlander-Lawson duality for all integers $s$ and $k\leq 2q$, where $n=\dim X$.
\end{proposition}
\begin{proof}
Let $\mathcal{D}:L^qH^k(X)\to L_{n-q}H_{2n-k}(X)$ be the the Friedlander-Lawson duality homomorphism (cf. \cite{Friedlander-Lawson2}).
We denote by the same notation for the coefficient extension map $\mathcal{D}:L^qH^k(X)_{\Q}\to L_{n-q}H_{2n-k}(X)_{\Q}$.
Since $\mathcal{D}$ is an isomorphism,  it is enough to show that  the image of $L^qH^k(X)_{\Q}^s$ under $\mathcal{D}$ is $L_{n-q}H_{2n-k}(X)_{\Q}^s$.
First note that there is a commutative diagram
\begin{equation}\label{eq13}
\xymatrix{L^qH^k(X)\ar[r]^-{\mathcal{D}}\ar[d]^{m_X^{!}}&L_{n-q}H_{2n-k}(X)\ar[d]^{m_{X*}}\\
L^qH^k(X)\ar[r]^-{\mathcal{D}}&L_{n-q}H_{2n-k}(X),
}
\end{equation}
where $m_X^{!}$ is the Gysin map induced by the map $m_X:X\to X$ (cf. \cite[Prop.5.5]{Friedlander-Lawson2}). By the Fulton's excess formula
for the morphic cohomology (also Lawson homology), we have $m_X^!m_X^*=\deg m_X$ (cf. \cite{Hu-Li}). The latter is equal to $m^{2n}$.

Now or $\alpha\in L^qH^k(X)_{\Q}^s$, by definition we have $m_{X}^*(\alpha)=m^{k-s}\alpha$. Hence $\alpha=\frac{1}{m^{k-s}}\cdot m_X^*\alpha$ and
\begin{equation}\label{eq14}
m_X^!\alpha=\frac{1}{m^{k-s}}\cdot m^!m_X^*\alpha=\frac{1}{m^{k-s}}\cdot m^{2n}\alpha=m^{2n-k+s}\alpha.
\end{equation}

From Equation (\ref{eq13}),  we have $\mathcal{D}\circ m_{X}^!=m_{X*}\circ \mathcal{D}$. So by this as well as Equation (\ref{eq14}),
 one gets
\begin{equation}\label{eq15}
m_{X*} (\mathcal{D}\alpha)=\mathcal{D}( m_{X}^!\alpha)= \mathcal{D}(m^{2n-k+s}\alpha)=m^{2n-k+s}\mathcal{D}\alpha.
\end{equation}
That is to say, $\mathcal{D}\alpha\in L_{n-q}H_{2n-k}(X)_{\Q}^s$. This gives us an isomorphism
\begin{equation*}
L^qH^k(X)_{\Q}^s\cong L_{n-q}H_{2n-k}(X)_{\Q}^s
\end{equation*}
for all integers $s$ and $k\leq 2q$.
\end{proof}

By this proposition, all the results in terms of Lawson homology in section \ref{sec6} have the corresponding morphic cohomological version obtained
by replacing $L^qH^k(X)_{\Q}^s$ by $L_{n-q}H_{2n-k}(X)_{\Q}^s$. For example, we have a decomposition for the morphic cohomology of an abelian variety.

\begin{proposition}\label{prop24}
Let $X$ be an abelian variety of dimension $n$. Then we have the following decomposition
 \begin{equation*}
 L^qH^k(X)_{\Q}\cong \bigoplus^{[\frac{k}{2}]}_{s=k-q-n}L^qH^k(X)_{\Q}^s.
 \end{equation*}
\end{proposition}
\begin{proof}
It follows from Theorem \ref{Th4} and Proposition \ref{prop23}.
\end{proof}

\section{Semi-topological K-theory}

Recall that the (singular) \textbf{semi-topological $K$-theory} (denoted by $K_*^{sst}(-)$)
was introduced and developed by Friedlander and Walker in a sequence of papers (cf.
\cite{Friedlander-Walker}, \cite{Friedlander-Walker2}, \cite{Friedlander-Walker3}, \cite{Friedlander-Walker4} and reference therein).

Let  ${\mathcal K}^{sst}(X)$ be a homotopy-theoretic group
completion  of a space of maps of $X$ to an
infinite Grassmannian, topologized as in \cite{Friedlander-Walker3}.
 The \textbf{semi-topological $K$-group} $ K_j^{sst}(X)$ of $X$
is defined to be the $j$-th homotopy group of ${\mathcal K}^{sst}(X)$.
The rational  $K^{sst}$-groups is denoted by
$$K_j^{sst}(X)_{\mathbb{Q}}:=K_j^{sst}(X)\otimes{\mathbb{Q}}.
$$

One of the fundamental result in semi-topological $K$-theory
is that there is a natural isomorphism between rational  $K^{sst}$-groups and
certain direct sum of rational morphic groups:
\begin{theorem}[Friedlander-Walker \cite{Friedlander-Walker3}]\label{Rat-K-iso}
There is a natural isomorphism
$$  K_j^{sst}(X)_{\mathbb{Q}}\cong \bigoplus_{q\geq 0} L^{q}H^{2q-j}(X)_{\mathbb{Q}}, \quad j\geq 0
$$
for any smooth  (quasi-)projective variety $X$.
\end{theorem}

Now let $X$ be an abelian variety. From Theorem \ref{Rat-K-iso} and Proposition \ref{prop23}, we have the following result.
\begin{corollary}\label{cor26}
Let $X$ be an abelian variety of dimension $X$. Then
$$
K_j^{sst}(X)_{\mathbb{Q}}\cong \bigoplus_{q\geq 0} \bigoplus^{[\frac{2q-j}{2}]}_{s=q-n-j}L^qH^{2q-j}(X)_{\Q}^s,
$$
In other words, $K_j^{sst}(X)_{\mathbb{Q}}$ is decomposed to be the direct sum of eigenspaces of the map $m^*_X:K_j^{sst}(X)_{\mathbb{Q}}\to K_j^{sst}(X)_{\mathbb{Q}}$.
\end{corollary}

\begin{proof} We have the following isomorphisms
$$
\begin{array}{cclr}
K_j^{sst}(X)_{\mathbb{Q}}&\cong& \bigoplus_{q\geq 0} L^{q}H^{2q-j}(X)_{\mathbb{Q}}&\hbox{(by Theorem \ref{Rat-K-iso})}\\
&\cong& \bigoplus_{q\geq 0} \bigoplus^{[\frac{2q-j}{2}]}_{s=q-n-j}L^qH^{2q-j}(X)_{\Q}^s.& \hbox{(by Proposition \ref{prop23} )}
\end{array}
$$
\end{proof}

Note that in Corollary \ref{cor26}, there is only finite many nonzero direct summands on the right side of the equation
since, for fixed $j$, $L^qH^{2q-j}(X)$ vanishes when $q$ large.  In particular,  for  an abelian variety of dimension three,
we give explicitly the following equations.

\begin{example} Let  $X$  be an abelian variety of $\dim X=3$. Then we have
\begin{equation}\label{eq16}
K_0^{sst}(X)_{\Q}\cong H^0(X,\Q)\oplus NS(X)_{\Q}\oplus L_1H_2(X)^0_{\Q}\oplus {\rm Griff}_1(X)_{\Q}\oplus H^6(X,\Q).
\end{equation}

\begin{equation}\label{eq17}
K_1^{sst}(X)_{\Q}\cong H^1(X,\Q)\oplus L_1H_3(X)^0_{\Q}\oplus  L_1H_3(X)^1_{\Q}\oplus H^5(X,\Q).
\end{equation}
and there is surjective map
\begin{equation}\label{eq18}
K_j^{sst}(X)_{\Q}\twoheadrightarrow K_j^{top}(X)_{\Q}, ~\forall j\geq 2.
\end{equation}

The weak version Suslin conjecture for Lawson homology with rational coefficients implies that
\begin{equation}\label{eq19}
K_j^{sst}(X)_{\Q}\cong K_j^{top}(X)_{\Q}, ~\forall j\geq 2
\end{equation}
where $K_j^{top}(X)_{\Q}$ is the $j$-th topological $K$-group with rational coefficients.
\end{example}
\begin{proof}
From Corollary \ref{cor26}, one has
$$
K_0^{sst}(X)_{\Q}\cong H^0(X,\Q)\oplus NS(X)_{\Q}\oplus  L^2H^4(X)_{\Q}^0\oplus L^2H^4(X)_{\Q}^1\oplus H^6(X,\Q).
$$
Note that $L^2H^4(X)_{\Q}^0\cong L_1H_2(X)^0_{\Q}$ is a finite dimensional $\Q$-vector space. Now Equation (\ref{eq16}) follows from
a fact that $L^2H^4(X)_{\Q}^1\cong L^2H^4(X)_{hom,\Q}={\rm Griff}_1(X)_{\Q}$ by Beauville (cf. \cite[Prop.6]{Beauville2}).
Equation (\ref{eq17}) follows from Corollary \ref{cor26} and Proposition \ref{prop23}.
Equation (\ref{eq18}) follows from Corollary \ref{cor26} and the surjectivity  of
$\Phi_{p,k}:L_pH_k(X)_{\Q}\to H_k(X)_{\Q}$ for all $k\geq n+p$.
Equation (\ref{eq19}) follows from  Corollary \ref{cor26}
the Atiyah-Hirzebruch isomorphism between the topological $K$-group with rational coefficients and the singular cohomology with rational coefficients,
and the fact the morphic cohomology is isomorphic to the singular cohomology for $X$ in these cases under the assumption of the weak Suslin conjecture.
\end{proof}

If we denote by $K_j^{sst}(X)_{\Q}^s=\{\alpha\in K_j^{sst}(X)_{\Q}| m^*_X(\alpha)=m^{s}\alpha\}$, then
$$\left\{
\begin{array}{lcl}
K_0^{sst}(X)_{\Q}^0&\cong&  H^0(X,\Q)\cong \Q;\\
K_0^{sst}(X)_{\Q}^2&\cong& NS(X)_{\Q};\\
K_0^{sst}(X)_{\Q}^3&\cong& {\rm Griff}^2(X)_{\Q}={\rm Griff}_1(X)_{\Q};\\
K_0^{sst}(X)_{\Q}^4&\cong& L^2H^4(X)^0\cong L_1H_2(X)^0;\\
K_0^{sst}(X)_{\Q}^6&\cong& H^6(X,\Q)\cong \Q;\\
K_0^{sst}(X)_{\Q}^s&=&0, \hbox{all other $s$}.
\end{array}\right.
$$

$$\left\{
\begin{array}{lcl}
K_1^{sst}(X)_{\Q}^1&\cong&  H^1(X,\Q);\\
K_1^{sst}(X)_{\Q}^3&\cong&  L^2H^3(X)^0_{\Q}\cong L_1H_3(X)^0_{\Q}; \\
K_1^{sst}(X)_{\Q}^4&\cong&  L^2H^3(X)^0_{\Q}\cong L_1H_3(X)^0_{\Q};\\
K_1^{sst}(X)_{\Q}^5&\cong&  H^5(X,\Q);\\
K_1^{sst}(X)_{\Q}^s &=&0,   \hbox{all other $s$}.
\end{array}\right.
$$

and

$$
K_j^{sst}(X)_{\Q}^s\cong H^s(X,\Q), \forall s\in \Z, j\geq 2.
$$


\section{Appendix}
In this appendix we will discuss and summary for convenience  the relations between a few  conjectures in  Lawson homology theory.
Some of them are known or implied in the literatures before.

Let $X$ be a smooth complex projective variety.
It was shown in \cite[\S 7]{Friedlander-Mazur} that the subspaces
$T_pH_k(X,{\mathbb{Q}})$ form a decreasing filtration (called the \emph{topological filtration}):
$$\cdots\subseteq T_pH_k(X,{\mathbb{Q}})\subseteq T_{p-1}H_k(X,{\mathbb{Q}})
\subseteq\cdots\subseteq
T_0H_k(X,{\mathbb{Q}})=H_k(X,{\mathbb{Q}})$$ and
$T_pH_k(X,{\mathbb{Q}})$ vanishes if $2p>k$.

Denote by
$G_pH_k(X,{\mathbb{Q}})\subseteq H_k(X,{\mathbb{Q}})$ the
$\mathbb{Q}$-vector subspace of $H_k(X,{\mathbb{Q}})$ generated by
the images of mappings $H_k(Y,{\mathbb{Q}})\rightarrow
H_k(X,{\mathbb{Q}})$, induced from all morphisms $Y\rightarrow X$ of
varieties of dimension $\leq k-p$.

The subspaces $G_pH_k(X,{\mathbb{Q}})$ also form a decreasing
filtration (called the \emph{geometric filtration}):
$$\cdots\subseteq G_pH_k(X,{\mathbb{Q}})\subseteq G_{p-1}H_k(X,{\mathbb{Q}})
\subseteq\cdots\subseteq G_0H_k(X,{\mathbb{Q}})\subseteq
H_k(X,{\mathbb{Q}})$$

Denote by $\tilde{F}_pH_k(X,{\Q})\subseteq
H_k(X,{\Q})$  the maximal sub-(Mixed) Hodge structure of span
$k-2p$. (See \cite{Grothendieck2} and \cite{Friedlander-Mazur}.) The sub-${\Q}$ vector spaces
$\tilde{F}_pH_k(X,{\Q})$ form a decreasing filtration of
sub-Hodge structures:
$$\cdots\subseteq \tilde{F}_pH_k(X,{\Q})\subseteq \tilde{F}_{p-1}H_k(X,{\Q})
\subseteq\cdots\subseteq \tilde{F}_0H_k(X,{\Q})\subseteq H_k(X,{\Q})$$ and $\tilde{F}_pH_k(X,{\Q})$ vanishes if $2p>k$. This
filtration is called the \emph{Hodge filtration}.

It was shown by Friedlander and Mazur that
\begin{equation}\label{eq20}
T_pH_k(X,{\mathbb{Q}})\subseteq G_pH_k(X,{\mathbb{Q}})\subseteq \tilde{F}_pH_k(X,{\Q})
\end{equation}
holds for any smooth projective variety $X$ and $k\geq 2p\geq0$.

Friedlander and Mazur proposed the following conjecture which relates Lawson homology theory to the central problems in the algebraic cycle theory.
\begin{conjecture}[Friedlander-Mazur conjecture, \cite{Friedlander-Mazur}]\label{conj9.1}
For any smooth projective variety $X$, one has
$$
T_pH_k(X,{\mathbb{Q}})= G_pH_k(X,{\mathbb{Q}}).
$$
\end{conjecture}

The Friedlander-Mazur conjecture remains open for general threefolds. However, it has been verified for some cases.
For example, it was shown to hold for general abelian
varieties (cf. \cite{Friedlander2}) or  abelian varieties for which the the generalized Hodge conjecture holds (cf. \cite{Abdulali});
It was also shown to hold for threefold $X$ with $h^{2,0}(X)=0$, in particular,
the complete intersection of dimension three (cf. \cite{Hu}).  It also holds for \emph{any} abelian threefold.
To see the last statement, we note that for $X$ a threefold $X$, it is enough to show Conjecture \ref{conj9.1} for the $p=1$.
Note that by Proposition \ref{Prop6.12}, $T_1H_k(X,{\mathbb{Q}})= G_1H_k(X,{\mathbb{Q}})$ for $k\geq 4$. By Proposition 1.15 in \cite{Hu},
one has $T_1H_3(X,{\mathbb{Q}})= G_1H_3(X,{\mathbb{Q}})$. Moreover, Friedlander and Mazur proved that
$T_1H_2(X,{\mathbb{Q}})= G_1H_2(X,{\mathbb{Q}})$ (cf. \cite[\S 7]{Friedlander-Mazur}).

\begin{conjecture}[The generalized Hodge conjecture, \cite{Grothendieck2} and \cite{Friedlander-Mazur}]\label{conj9.2}
For any smooth projective variety $X$, one has
$$
G_pH_k(X,{\mathbb{Q}})=\tilde{F}_pH_k(X,{\Q}).
$$
\end{conjecture}

There is a corresponding conjecture in terms of morphic cohomology (cf. \cite{Friedlander-Walker4}).
The equivalence between the homological version and the cohomological
version is given by using Friedlander-Lawson duality isomorphism (cf. \cite{Friedlander-Lawson2}). In  \cite{Friedlander-Walker4},
the cohomological version of Conjecture \ref{conj9.1} and \ref{conj9.2} combine to one conjecture $T_pH_k(X,{\mathbb{Q}})=\tilde{F}_pH_k(X,{\Q})$
which they called the (strong) Friedlander-Mazur conjecture.

Now let $X\subset \P^N$ be a smooth variety of dimension $n$ and let $H$ be a hyperplane section such that $Y:=X\cap H$ is smooth.
The Lefschetz operator $L: H^i(X,\Q)\to H^{i+2}(X,\Q)$ is defined by $L(\alpha)=\alpha\cup [Y]$, where $[Y]$ denotes the homology class of
$Y$ in $H^2(X,\Q)$. The Hard Lefschetz Theorem says
$$
L^{n-i}:H^i(X,\Q)\to H^{2n-i}(X,\Q)
$$
is an isomorphism. Note that this $L_X$ is given by the algebraic cycle $\Delta(Y)$, in other words, $L(-)=p_{2*}\big(p_1^*{(-)}\cap \Delta(Y)\big)$.

\begin{conjecture}[Grothendieck standard conjecture of Lefschetz type, \cite{Grothendieck}]
The inverse $\Lambda^{n-i}:H^{2n-i}(X,\Q)\to H^i(X,\Q)$ to  $L^{n-i}$ is given by an algebraic cycle for each $0\leq i\leq n$.
\end{conjecture}

In this case  $\Lambda^{n-i}$ is also called \emph{algebraic} for each $0\leq i\leq n$.

\begin{conjecture}[Hard Lefschetz conjecture for Lawson homology, \cite{Friedlander-Mazur}]\label{conj9.04}
Let $X$ be a smooth projective variety of dimension $n$ and let $h$ be a hyperplane section. Then the intersection
$$
h^k\bullet:L_pH_{n+k}(X)_{\Q}\to L_{p-k}H_{n-k}(X)_{\Q}, ~\alpha\mapsto[h]^{k} \bullet \alpha,
$$
is injective for $k\geq 1$, where $[h]$ is viewed as the class of $h$  in $L_{n-1}H_{2n-2}(X)_{\Q}$.
\end{conjecture}

The following conjecture is essential to  the structure of the Lawson homology for a smooth quasi-projective variety.

\begin{conjecture}[The Suslin conjecture for Lawson Homology with coefficient $A$, \cite{Friedlander-Haesemesyer-Walker}]\label{conj9.5}
For any abelian group $A$ and any smooth quasi-projective variety $X$ of dimension $n$,
the map $\Phi_{p,k}: L_pH_k(X,A) \to H_k(X,A)$ is an isomorphism for $k\geq n+p$ and is an injection for $k=n+p-1$.
\end{conjecture}

This is an analogue of the Beilinson-Lichtenbaum conjecture in motivic cohomology theory.
When $A$ is a finite abelian group, the Milnor-Bloch-Kato conjecture(or theorem) implies that
the Suslin conjecture for Lawson Homology with coefficient $A$.

The weak version  Suslin conjecture for Lawson homology only requires $X$ be projective and  $\Phi_{p,k}$ be isomorphic for $k\geq n+p$. Clearly,
the weak version  Suslin conjecture is much weaker than the (strong) Suslin conjecture. However, it is  mentioned separately
since the weak version Suslin conjecture  for Lawson homology with integer coefficients
implies  the  Grothendieck standard conjecture of Lefschtez type (see below).

\medskip
These conjectures (Conjecture \ref{conj9.1}-\ref{conj9.5}) hold for all smooth projective varieties of dimension less or equal than two.
However, as far as I know,  all of them are still open even for a general smooth projective variety of dimension three.
Known cases  and statements on these conjectures can be found in the literature (cf. \cite{Beilinson}, \cite{Friedlander2},  \cite{Friedlander-Haesemesyer-Walker}, \cite{Friedlander-Mazur}, \cite{Friedlander-Walker4}), \cite{Hu}, \cite{Lima-Filho}, \cite{Peters}, \cite{Xu},
\cite{Voineagu}, etc. and the references therein).

\begin{lemma}\label{lemma9.6}
The maps $\Phi_{p,k}: L_pH_k(X,\Q) \to H_k(X,\Q)$ are surjective for all smooth projective variety $X$ and $k\geq p+\dim X$
is equivalent to the Friedlander-Mazur conjecture holds for all smooth projective variety.
\end{lemma}
\begin{proof}
On one side, it is clear that if the Friedlander-Mazur conjecture for a smooth projective variety $X$, then $T_pH_k(X,\Q)=G_pH_k(X,\Q)$ for
$k\geq p+\dim X$. Since $k\geq p+\dim X$, $G_pH_k(X,\Q)=H_k(X,\Q)$ and so $T_pH_k(X,\Q)=H_k(X,\Q)$, i.e.,
$\Phi_{p,k}: L_pH_k(X,\Q) \to H_k(X,\Q)$ is surjective.

One the other side,  we need to show that for any smooth projective variety $Y$, $T_pH_k(Y,\Q)=G_pH_k(Y,\Q)$ for all $k\geq 2p$.
 By assumption, we only need to show  $T_pH_k(Y,\Q)=G_pH_k(Y,\Q)$ for all $2p\leq k\leq n+p-1$. This was done in the proof
 of Theorem 1.20 in \cite{Hu}, where we base on a stronger  assumption (i.e. the Suslin conjecture). However,
 the injection for $\Phi_{p,k}$ is not really used in that proof.
\end{proof}

\begin{proposition}
The Friedlander-Mazur conjecture is equivalent to the Grothen- dieck standard conjecture of Lefschtez type. More precise,
the Friedlander-Mazur conjecture holds for all smooth projective varieties  if and only if  $\Lambda$ is algebraic
for every smooth projective varieties.
\end{proposition}
\begin{proof}
On one side, we note that the Grothendieck standard conjecture of Lefschtez type implies the Friedlander-Mazur conjecture (cf. \cite{Friedlander2}).
On the other side, the Friedlander-Mazur conjecture implies that $\Phi_{p,k}: L_pH_k(X,\Q) \to H_k(X,\Q)$ is surjective for all $k\geq p+\dim X$.
The surjectivity of $\Phi_{p,k}$ for  all $k\geq p+\dim X$ is equivalent to  the Grothendieck standard conjecture of
Lefschtez type (cf. \cite{Beilinson}).

\end{proof}

\begin{proposition}
The generalized Hodge conjecture holds for all smooth projective varieties implies the  Conjecture \ref{conj9.1}. More precisely,
$$
``G_pH_k(X,{\mathbb{Q}})=\tilde{F}_pH_k(X,{\Q}), ~ all~ X" \Leftrightarrow ``T_pH_k(Y,{\mathbb{Q}})=\tilde{F}_pH_k(Y,{\Q}), ~all~ Y".
$$
\end{proposition}

\begin{proof}
The part $``\Leftarrow"$  follows directly from the assumption and Equation (\ref{eq20}).
For the part $``\Rightarrow"$, we assume the generalized Hodge conjecture holds
for all smooth projective varieties, in particular, it holds for $X\times X$ in the case that $k=2p$, which is the classical Hodge conjecture.
It is known the Hodge conjecture for $X\times X$ implies that the Grothendieck standard conjecture for $X$, which holds for all smooth
projective varieties implies that Conjecture \ref{conj9.1} holds for all smooth projective varieties.
\end{proof}



\begin{remark}
From the above discuss we see that the generalized Hodge conjecture and the Suslin conjecture for Lawson homology with integer coefficients dominate
many key problems in the theory of Lawson homology. So solutions to those problems for general smooth projective varieties is
 the most difficult problems in this field. An alternative way  to deal with problems in  Lawson homology theory
  would be the study on those varieties carrying special structures.
\end{remark}

\section*{Acknowledgements}
I would like to thank Baohua Fu for useful conversions during my visiting Chinese Academy of Science
during December 2010. The author is  grateful the Max-Planck Institute for Mathematics at Bonn for its hospitality
and financial support during my visiting period.
The project was partially sponsored by SRF for ROCS, SEM and  NSFC(11171234).

\end{document}